\documentclass[12pt,reqno]{amsart}
\usepackage[T1]{fontenc}
\usepackage{amsfonts}
\usepackage{amssymb}
\usepackage{mathrsfs}
\usepackage[latin1]{inputenc}
\usepackage{amsmath}
\usepackage[english]{babel}
\usepackage{setspace}

\newtheorem{thm}{Theorem}[section]
\newtheorem{cor}[thm]{Corollary}

\newtheorem{prop}[thm]{Proposition}
\theoremstyle{definition}

\newtheorem{rem}[thm]{Remark}

\numberwithin{equation}{section}

\numberwithin{equation}{section}

\newcommand{\R}{{\mathbb R}}

\newcommand{\C}{{\mathbb C}}
\newcommand{\N}{{\mathbb N}}

\newcommand{\cD}{{\mathcal D}}
\newcommand{\cL}{{\mathcal L}}

\newcommand{\cS}{{\mathcal S}}
\newcommand{\cO}{{\mathcal O}}

\newcommand{\su}{\subseteq}

\newcommand\proj{\mathop{\rm proj\,}}
\newcommand\ind{\mathop{\rm ind\,}}

\begin{document}

%
%
%
%

\title[Dynamics of composition operators\dots]
 {Dynamics of composition operators on function spaces defined by local and global properties }

\author[A.A. Albanese, E. Jord\'a, C. Mele]{Angela\,A. Albanese, Enrique Jord\'a and Claudio Mele}

\address{ Angela A. Albanese\\
Dipartimento di Matematica e Fisica ``E. De Giorgi''\\
Universit\`a del Salento- C.P.193\\
I-73100 Lecce, Italy}
\email{angela.albanese@unisalento.it}

\address{Enrique Jord\'a\\ Instituto Universitario de Matem\'atica Pura y Aplicada IUMPA\\ Universitat Polit\`ecnica de Val\`encia,
EPS Alcoy Plaza Ferrándiz y Carbonell, s/n, 03801 Alcoy (Alicante), Spain.}
\email{ejorda@mat.upv.es}

\address{Claudio Mele\\
	Dipartimento di Matematica e Fisica ``E. De Giorgi''\\
	Universit\`a del Salento- C.P.193\\
	I-73100 Lecce, Italy}
\email{claudio.mele1@unisalento.it}





\begin{abstract} In this paper we consider composition operators on locally convex spaces of functions defined on $\R$. We prove results concerning supercyclicity, power boundedness, mean ergodicity and convergence of the iterates in the strong operator topology.
\end{abstract}

\maketitle
\section{Introduction }\label{intro}
In this paper we deal with composition operators  defined in locally convex Hausdorff spaces of  real valued continuous functions. We study when the operator has \emph{big orbits}, i.e., when it is supercyclic or hypercyclic, and also when it has \emph{small orbits}, i.e., when the operator is power bounded or mean ergodic. We collect in this section our results and relate them with recent literature.

We focus our attention on composition operators on locally convex  Hausdorff spaces $X\hookrightarrow C^m(\R)$, for $m\in\N_0\cup \{\infty\}$. For a function $\varphi\in C^{m}(\R)$, being increasing and with no fixed points is equivalent to being {\em strongly runaway}, i.e.,  for all compact subset $K\subset \R$  there is $n_0\in\N$ such that $\varphi^n(K)\cap K=\emptyset$ for all $n\geq n_0$ (see \cite[Lemma 4.1]{Pr2} for a proof).
 Kalmes got in \cite[Corollary 4.1, Corollary 4.2]{Kalmes2}  the equivalence between hypercyclicity and  strongly runaway in a more general context, which includes weighted composition operators defined in function spaces of  several variables.  The results of Kalmes generalize those of Przestacki obtained in \cite[Lemma 4.1, Theorem 4.5, Theorem 4.6]{Pr2} for composition operators on  $C^\infty(\Omega)$, with $\Omega\subseteq\R^d$ open.   Our main result here is that, under natural assumptions on $X$, $C_\varphi$ being weakly supercyclic implies that $\varphi$ is increasing, with no vanishing derivative when $m>0$, and without fixed points. This, using previous works of Bonet and Doma\'nski and Kalmes, permits us to conclude that weak supercyclicity is equivalent to being mixing for composition operators defined on $C^m(\R)$ with $m\in \N\cup \{\infty\}$ and that weak supercyclicity implies to be mixing for composition operators defined   on $\mathcal{A}(\R)$.

We also obtain results for power boundedness and mean ergodicity of composition operators. We prove that for a Montel space $X\hookrightarrow C^1(\R)$  satisfying that, for every symbol $\varphi\in X$
\begin{itemize}
\item[(*)]   The composition operator $C_\varphi:X\to X$ is power bounded if, and only if, the sequence $\{\varphi_n\}_{n\in\N}$  of iterates of $\varphi$ is bounded in $X$,
\end{itemize}
under natural assumptions on $X$, choosing a monotonic $\varphi\in X$ with $\varphi_2(x)\neq x$, then the composition operator $C_\varphi$ is power bounded if, and only if, $\varphi$ has a unique  fixed point $a$ on which the sequence of the iterates $\varphi_n$ of $\varphi$ converges if, and only if, the sequence $\{C_{\varphi_n}\}_{n\in\N}$ is convergent in $\cL_s(X)$ to $C_{a}$ ($C_{a}:X\to X, f\mapsto f(a)$).

Furthermore, if we restrict to particular locally convex spaces $X$ and a polynomial $\varphi(x)$, then the composition operator is equivalently power bounded and mean ergodic if, and only if,  $\varphi(x)=ax+b$, $|a|<1$. The excluded cases $\varphi(x)=x$ and $\varphi(x)=-x+b$ produce trivially power bounded uniformly mean ergodic composition operators. The equivalence between power boundedness, mean ergodicity and the convergence of the iterates of $\varphi$ to the unique fixed point of $\varphi$ has been proved in \cite{BD} to be true for $\mathcal{A}(\R)$, with no restriction on the symbol. Besides $\mathcal{A}(\R)$, our result applies to $C^\infty(\R)$,  which satisfies (*) by \cite{Kalmes} and to the space $\cO_M(\R): =\proj_{\stackrel{m}{\leftarrow}}\ind_{\stackrel{n}{\to}}\cO_{n}^m(\R)$ of the multipliers of the classical  space $\cS(\R)$.

The structure of the paper is the following: Section 2 is devoted to preliminaries and to the study of the symbols in $\cO_M(\R)$ and $\cO^m(\R)$, $m\in\N$. 
In Section 3 we study weakly supercyclic composition operators. In Section 4 we obtain respectively the results on dynamics explained above. Finally, Section 5 is devoted to show that $\cO_M(\R)$ and $\cO^m(\R)$, $m\in\N$, satisfy condition (*).

\section{Preliminaries and the action of the composition operator on $\cO_M(\R)$}
\subsection{Preliminaries}
Throught the paper, we denote by $\N$  the set of all positive integer numbers and set $\N_0:=\N\cup \{0\}$.

 Let X be a locally convex Hausdorff space (briefly, lcHs). We denote by $\cL(X)$ 
the space of all continuous linear operators from $X$ into itself.

An operator $T\in \cL(X)$, with $X$ a lcHs, is called \textit{power bounded} if $\{T^n\}_{n\in\N}$ is an
equicontinuous subset of $\cL(X)$.

The  Ces\`aro means of an operator $T\in\cL(X)$, with $X$ a lcHs, are defined by
\[
T_{[n]} :=\frac{1}{n}\sum_{m=1}^nT^m,\quad n\in\N.
\]
The operator  $T$ is called \textit{mean ergodic} (resp. \textit{uniformly mean ergodic}) if $\{T_{[n]}\}_{n\in\N}$ 
is a convergent
sequence in $\cL_s(X)$ (resp. in $\cL_b(X)$), where $\cL_s(X)$ denotes $\cL(X)$ endowed with the strong operator topology $\tau_s$ (resp. $\cL_b(X)$ denotes $\cL(X)$ endowed with the topology $\tau_b$  of the uniform convergence on bounded subsets of $X$). 

For reflexive Fr\'echet spaces (in Montel Fr\'echet spaces, resp.) every power
bounded operator is necessarily mean ergodic (uniformly mean ergodic, resp.) by \cite[Corollary 2.7, Proposition 2.9]{ABR-0}. There exist mean ergodic operators which are not power bounded, see, f.i., \cite[\S 6]{H}.
For further results on mean ergodic operators we refer to \cite{Kr,Y}. For recent results on mean ergodic operators in
lcHs' we refer to \cite{ABR-0,ABR-JJ,ABR-2,AM,BJ,FeGaJo,Kalmes}, for example, and the references therein.

Given an operator $T\in \cL(X)$, we denote by $\sigma(T)$ the \textit{spectrum} of $T$, i.e., the subset of $\C$ formed by the $\lambda$ such that $\lambda I-T$ is not invertible, and by $\sigma_p(T)$ its \textit{point spectrum}, i.e., the subset of $\sigma(T)$ formed by those $\lambda$ such that $\lambda I-T$ is not injective.
\\[0.1cm]
Moreover, we recall that  $T'$ denotes as usual the dual operator of $T$. We have that $T\in \cL(X'_\beta)$ and also $T'\in \cL(X'_\sigma)$, where $X'_\beta$ and $X'_\sigma$ stand for the topological dual of $X$ endowed with the strong topology $\beta(X',X)$ and the weak topology $\sigma(X',X)$ resp.

Let  $X$ be a lcHs of functions defined on $\R$ and let $\varphi\colon \R\to \R$ be a function. If $f\circ \varphi\in X$ for all $f\in X$, then we can consider the composition operator $C_\varphi\colon X\to X$, $f\mapsto f\circ \varphi$. The operator $C_\varphi$ is clearly linear. In case that $C_\varphi\in \cL(X)$, the function $\varphi$ is said to be a \textit{symbol} for $X$.

Given a function $\varphi\colon\R\to\R$,  we denote by $\varphi_n$ the $n$-th iterate of $\varphi$ for all $n\in\N$, i.e., $\varphi_n:=\varphi\circ\ldots\circ\varphi$ $n$-times, and define $\varphi_{[n]}:=\frac{1}{n}\sum_{k=1}^n\varphi_k$ for all $n\in\N$. In case the function $\varphi$ is bijective and hence there exists $\varphi^{-1}\colon\R\to\R$, we denote by $\varphi_{-n}$ the $n$-th iterate of $\varphi^{-1}$ for all $n\in\N$.
We also recall that  a function $\varphi\colon\R\to\R $ is said to be increasing (decreasing, resp.) if $\varphi(x)<\varphi(y)$ ($\varphi(x)>\varphi(y)$, resp.) for every  $x,y\in\R$ such that  $x<y$. The function $\varphi$ is said to be non decreasing (non increasing, resp.) if $\varphi(x)\leq\varphi(y)$ ($\varphi(x)\geq\varphi(y)$, resp.) for every  $x,y\in\R$ such that  $x<y$.

In what follows, we recall the necessary definitions and some basic properties of the spaces $\cS(\R)$ and $\cO_M(\R)$. For general information about function spaces in the theory of distributions we refer to \cite{Ch,Sh}.

The space $\cS(\R)$ of \textit{rapidly decreasing functions} is a nuclear Fr\'echet space and hence, it is Montel  and reflexive. Accordingly, its strong dual $\cS'(\R)$ is  a nuclear lcHs. In particular, $\cS'(\R)$ is a barrelled and bornological lcHs. 

The space   $\cO_M(\R)$  of \textit{slowly increasing functions} on $\R$ is given by
\[
\cO_M(\R)=\cap_{m=1}^\infty\cup_{n=1}^\infty\cO^{m}_{n}(\R),
\]
where $\cO_{n}^m(\R):=\{f\in C^{m}(\R)\colon |f|_{m,n}:=\sup_{x\in\R}\sup_{0\leq i\leq m}(1+x^2)^{-n}|f^{(i)}(x)|<\infty\}$, endowed with the norm $|\cdot |_{m,n}$, is a Banach space for any $m,n\in\N$.  
The space $\cO_M(\R)$, endowed with its natural lc-topology, i.e., $\cO_M(\R): =\proj_{\stackrel{m}{\leftarrow}}\ind_{\stackrel{n}{\to}}\cO_{n}^m(\R)$, is a projective limit of the complete (LB)-spaces $\cO^m(\R):=\ind_{\stackrel{n}{\to}}\cO_{n}^m(\R)$, for $m\in\N$. In particular, $\cO_M(\R)$ is a bornological nuclear lcHs (hence, Montel and reflexive), and it is continuously embedded in $C^\infty(\R)$ (see \cite{Gro}).
Furthemore, a fundamental system of continuous norms on $\cO_M(\R)$ is given by
\[
p_{m,v}(f)=\sup_{x\in\R}\sup_{0\leq i\leq m}|v(x)||f^{(i)}(x)|,\quad f\in \cO_M(\R),
\]
where $v\in \cS(\R)$ and $m\in\N$ (see, f.i., \cite{Ch}).

The space $\cO_M(\R)$ is the space of multipliers of $\cS(\R)$ and its strong dual $\cS'(\R)$. So, for any fixed $h\in\cO_M(\R)$, the multiplication operator $M_h\colon \cS(\R)\to \cS(\R)$, $f\mapsto hf$, is continuous.

The spaces $C^m(\R)$, $m\in\N_0\cup\{\infty\}$, are always endowed with their natural lc-topology $\tau_c$ given from the canonical fundamental sequence of seminorms. We point out that $(C^m(\R),\tau_c)$ is a Fr\'echet space. 
In particular, $(C^\infty(\R),\tau_c)$ is a nuclear Fr\'echet space. Moreover, we denote by  $\cD^m(\R):=\{f\in C^m(\R)\, :\, {\rm supp}f\,\ {\rm is\, compact} \}$, for $m\in\N_0\cup\{\infty\}$ (we write $\cD(\R)$ for $m=\infty$). Then $\cD^m(\R)=\cup_{k=1}^\infty C^m_0([-k,k])$ for all $m\in\N$ and so, endowed  with its natural lc-topology, i.e., $\cD^m(\R):=\ind_{\stackrel{k}{\to}}C^m_0([-k,k])$, is a complete (LB)-space ((LF)-space, for $m=\infty$). 


We remark the following fact.

\begin{prop}
\label{schurr}
Let $m\in\N$ and let $B$ be a bounded subset of $\cO^m(\R)$. Then the spaces  $\cO^m(\R)$ and $C^m(\R)$ induce on $B$ the same topology.
\end{prop}
\begin{proof}
For any $p\in\N$ we will show that the spaces $\cO_{p+1}^{m}(\R)$ and $C^{m}(\R)$ induce the same topology  on the unit ball $B_p^m$ of $\cO_p^m(\R)$. To do this, we only need to show that for each $\varepsilon>0$, 
$U:=\{f\in B_p^m:\ |f|_{m,p+1}<\varepsilon\}$ is a neighborhood of $0$ in $B_p^m$ endowed with the topology inherited from $C^m(\R)$. So, let $M>0$ be such that $(1+x^2)^{-1}<\varepsilon$, whenever $|x|>M$ and let $ V:=\{f\in C^m(\R):\ |f^{(i)}(x)|<\varepsilon \text{ for all } i=0,\ldots, m,\ |x|\leq M\}$. Then  $B_p^m\cap V\subseteq U$. Indeed, if $f\in B_p^m\cap V$, then 
\[
(1+x^2)^{-p-1}|f^{(i)}(x)|\leq |f^{(i)}(x)|<\varepsilon, \quad i=0,\ldots, m, \ |x|\leq M,
\]
and 
{\small \[
(1+x^2)^{-p-1}|f^{(i)}(x)|=(1+x^2)^{-p}|f^{(i)}(x)|(1+x^2)^{-1}\leq (1+x^2)^{-1}<\varepsilon,\, i=0,\dots,m,\ |x|>M. 
\]}
Therefore, $f\in U$. The conclusion follows from the fact that the bounded subsets of $\cO^m(\R)$ are localized, i.e., $\cO^m(\R)$ is a regular (LB)-space.
\end{proof}

\begin{rem}\label{convom}
 As a consequence of Proposition \ref{schurr}, we get that a sequence $\{\varphi_n\}_{n\in\N}\subset \cO^m(\R)$ is convergent in $\cO^m(\R)$ to some $\varphi\in \cO^m(\R)$ if, and only if, $\{\varphi_n\}_{n\in\N}$  is a bounded sequence in  $\cO^m(\R)$ and $\lim_{n\to\infty} \varphi_n=\varphi$ in $C^m(\R)$. Therefore, as  $\cO_M(\R)$ is the projective limit of the complete (LB)-spaces $\cO^m(\R)$, for $m\in\N$, we can then conclude  that a sequence $\{\varphi_n\}_{n\in\N}\subset \cO_M(\R)$ is convergent in $\cO_M(\R)$ to some $\varphi\in \cO_M(\R)$ if, and only if, $\{\varphi_n\}_{n\in\N}$  is a bounded sequence in  $\cO_M(\R)$ and $\lim_{n\to\infty} \varphi_n=\varphi$ in $C^\infty(\R)$. 
 Therefore, for every sequence $\{f_n\}_{n\in\N}\subset\cD(\R)$ satisfying the condition for every compact subset $K$ of $\R$  there exists $n_0\in\N$ such that ${\rm supp}f_n\cap K=\emptyset$ for all $n\geq n_0$, we get that $\{f_n\}_{n\in\N}$ is convergent to $0$ in $\cO^m(\R)$  (in $\cO_M(\R)$, resp.) if, and  only if, it is a bounded sequence of $\cO^m(\R)$ (of $\cO_M(\R)$, resp.). 
 
 We point out that there are examples of sequences $\{f_n\}_{n\in\N}\subset\cD(\R)$ satisfying the previous condition but not bounded in $\cO^m(\R)$ (in $\cO_M(\R)$, resp.). Indeed, for all $n\in\N$ let $f_n\in \cD(\R)$ be such that $f_n(x):=n(1+x^2)^n$ for $x\in K_n:=\left[n+\frac{1}{n}, n+1-\frac{1}{n}\right]$ and supp$f_n\su [n,n+1]$. Then for any $m, p\in\N$ we have
  \begin{align*}
  |f_n|_{m,p}&=\sup_{x\in\R}\sup_{0\leq i\leq m}(1+x^2)^{-p}|f_n^{(i)}(x)|\\
  &\geq \sup_{x\in K_n}(1+x^2)^{-p}|f_n(x)|=n\sup_{x\in K_n}\frac{(1+x^2)^n}{(1+x^2)^p}\geq n, \ {\rm for\ all }\ n\geq p.
  \end{align*}
  This means that the sequence $\{f_n\}_{n\in\N}$ cannot be bounded in the Banach space $\cO^m_p(\R)$ for all $p\in\N$ and hence in $\cO^m(\R)$, being  $\cO^m(\R)=\ind_{\stackrel{p}{\to}}\cO^m_p(\R)$  a regular (LB)-space.
\end{rem}

\subsection{The composition operator on $\cO_M(\R)$}

In \cite{GaJo} the symbols $\varphi\in C^\infty(\R)$ for  $\cS(\R)$ have been characterized. Precisely,  \cite[Theorem  2.3]{GaJo} states that 
	a function $\varphi\in C^\infty(\R)$ is a symbol for $\cS(\R)$ if, and only if, the following conditions are satisfied:
	\begin{enumerate}
		\item For all $j\in\N_0$ there exist $C,p>0$ such that, for every $x\in\R$
		\begin{equation*}\label{symb1}
		|\varphi^{(j)}(x)|\leq C(1+\varphi(x)^2)^p;
		\end{equation*}
		\item There exists $k>0$ such that 
		\begin{equation*}\label{symb2}
		|\varphi(x)|\geq |x|^{\frac{1}{k}}, \quad \forall \,|x|\geq k.
		\end{equation*}
	\end{enumerate}


We now characterize the function $\varphi\in C^\infty(\R)$  such that $C_\varphi$ acts continuously from $\cO_M(\R)$ into itself. In order to do this, we first establish the following result.

\begin{thm}\label{thcomp} Let $m\in\N$ and let  $\varphi\in C^m(\R)$.
	Then the following properties are equivalent:
	\begin{enumerate}
		\item $C_\varphi$ acts continuously from $\cO^m(\R)$ into itself;
		\item  $\varphi\in\cO^m(\R)$.
	\end{enumerate} 
\end{thm}

\begin{proof}
	$(1)\Rightarrow(2)$. Consider the function $p_1(x):=x$, for $x\in\R$, which belongs to $\cO^m(\R)$. Then $\varphi=C_\varphi p_1 \in \cO^m(\R)$. So, (2) is satisfied.	

	$(2)\Rightarrow(1)$. 
	We claim that  for all $n\in\N$ there exists $n'\in\N$ such that $C_\varphi: \cO_n^m(\R)\to \cO_{n'}^m(\R)$ is continuous. Indeed, for  fixed $n\in\N$ and  $f\in \cO_n^m(\R)$, we observe  for every $x\in\R$ and $i=0,\dots,m$ that 
	\begin{align}\label{eq.derivate}
	&(C_\varphi f)^{(i)}(x)=\nonumber\\
	&=\sum\frac{i!}{k_1!k_2!\ldots k_n!}f^{(k)}(\varphi(x))\left(\frac{\varphi'(x)}{1!}\right)^{k_1}\left(\frac{\varphi''(x)}{2!}\right)^{k_2}\ldots \left(\frac{\varphi^{(i)}(x)}{i!}\right)^{k_i},
	\end{align}
	where the sum is extended over all $(k_1,k_2,\ldots,k_i)\in \N_0^i$ such that $k_1+2k_2+\ldots+ ik_i=i$ and $k_1+k_2+\ldots+k_i=k$ (observe that $k\leq i$). Since by assumption on $\varphi$, i.e., $\varphi\in \cO^m(\R)$, there exist $D>0$ and $p\in\N$ such that $|\varphi^{(j)}(x)|\leq D(1+x^2)^p$ for $x\in\R$ and $j=0,\ldots, m$ and hence $(1+\varphi(x)^2)\leq D'(1+x^2)^{2p}$ for $x\in\R$ and a suitable constant $D'>0$, from \eqref{eq.derivate} it follows for every  $x\in\R$ and $i=0,\ldots, m$ that 
	\begin{align*}
	|(C_\varphi f)^{(i)}(x)|&\leq\sum\frac{i!}{k_1!k_2!\ldots k_i!}|f^{(k)}(\varphi(x))|\left|\frac{\varphi'(x)}{1!}\right|^{k_1}\left|\frac{\varphi''(x)}{2!}\right|^{k_2}\ldots \left|\frac{\varphi^{(i)}(x)}{i!}\right|^{k_i}\\& \leq \sum\frac{i!}{k_1!k_2!\ldots k_i!}|f|_{m,n}(1+\varphi(x)^2)^n\frac{ D^k(1+x^2)^{p k}}{1!^{k_1}2!^{k_2}\ldots i!^{k_i}}\\&\leq C_i |f|_{m,n} (1+x^2)^{p(2n+i)}\leq C|f|_{m,n} (1+x^2)^{p(2n+m)},
	\end{align*}
	where $C_i:=(D')^{n}\sum\frac{i!}{k_1!k_2!\ldots k_i!}\frac{D^k}{1!^{k_1}2!^{k_2}\ldots i!^{k_i}}$ and $C:=\max_{i=0,\ldots,m}C_i$. Accordingly,  if  $n':=p(2n+m)\in\N$, we get that 
	\begin{align*}
	|C_\varphi f|_{m,n'}\leq C |f|_{m,n}.
	\end{align*}
	Since $n\in\N$ and  $f\in\cO_n^m(\R)$ are arbitrary, the claim is proved.
\end{proof}

As a consequence of the result above, we obtain the class of symbols for which $C_\varphi \in \cL(\cO_M(\R))$.

\begin{thm}\label{chasymom}
	Let $\varphi\in C^\infty(\R)$.
	Then the following properties are equivalent:
	\begin{enumerate}
		\item $C_\varphi$ acts continuously from $\cO_M(\R)$ into itself;
		\item $C_\varphi$ acts continuously from $\cO^m(\R)$ into itself for all $m\in\N$;
		\item  $\varphi\in\cO_M(\R)$.
	\end{enumerate} 
\end{thm}
\begin{proof}
	$(1)\Rightarrow(3)$. Consider the function $p_1(x):=x$, for $x\in\R$, which belongs to $\cO_M(\R)$. Then $\varphi=C_\varphi p_1 \in \cO_M(\R)$. So, (3) is satisfied.

	$(2)\Leftrightarrow (3)$. Follows by Theorem \ref{thcomp}.

	$(2)\Rightarrow(1)$. If the composition operator $C_\varphi$ acts continuously from $\cO^m(\R)$ into itself for all $m\in\N$, taking into account of the fact that $\cO_M(\R)$ is the projective limit of the (LB)-spaces $\cO^m(\R)$, it follows that the operator $C_\varphi$ acts continuously from $\cO_M(\R)$ into itself.
	\end{proof}

\begin{rem}
	Let $\varphi \in \cO_M(\R)$. If $\varphi$ is not a constant function, then the continuous linear  operator $C_\varphi: \cO_M(\R)\to\cO_M(\R)$ is never compact. The result follows from \cite[Theorem 3.3]{GaJo}, after having observed that  $\cD(\R)\hookrightarrow \cO_M(\R) \hookrightarrow C^\infty(\R)$ continuously.
\end{rem}

\begin{rem} Let $\cO_C(\R):=\cup_{n=1}^\infty\cap_{m=1}^\infty\cO_{n}^m(\R)$ be the space of \textit{very slowly increasing functions} on $\R$, endowed with its natural lc-topology, i.e., $\cO_C(\R) =\ind_{\stackrel{n}{\to}}\proj_{\stackrel{m}{\leftarrow}}\cO_{n}^m(\R)$, which is a  complete (LF)-space. In particular, $\cO_C(\R)$ is continuously included in  $\cO_M(\R)$ and it is the predual of the space of convolutors of the spaces $\cS(\R)$ and $\cS'(\R)$.

 Let $\varphi\in C^\infty(\R)$. If $C_\varphi$ acts continuously on $\cO_C(\R)$, then $\varphi\in\cO_C(\R)$ since $\varphi=C_\varphi p_1$, where $p_1\in\cO_C(\R)$ is the identity function. But the converse is not true.  For example, let $\varphi(x)=x^2$ and $f(x)=\sin x$, for $x\in\R$. Then $\varphi, f\in \cO_C(\R)$. But $f\circ \varphi\not\in \cO_C(\R)$. Indeed, we have for every $x\in\R$ and  $i\in\N$, with $i$ even, that  
	\[
	(f\circ \varphi)^{(i)}(x)=P(x)\sin x^2+Q(x)\cos x^2,
	\] 	
	where $P$ is a polynomial of degree equal to $i$ and $Q$ is a polynomial of degree  less than $i$. So, $f\circ \varphi\not\in \cap_{m=1}^\infty \cO^m_n(\R)$ for all $n\in\N$. Otherwise, if $f\circ \varphi\in \cap_{m=1}^\infty \cO^m_{n_0}(\R)$ for some $n_0\in\N$, then 
	\[
	|(f\circ \varphi)|_{m,n_0}=\sup_{x\in\R}\sup_{j=0,\ldots, m}(1+x^2)^{-n_0}|(f\circ \varphi)^{j}(x)|<+\infty
	\] 
	for all $m\in\N$. But if we choose  $x^2_k:=\frac{\pi}{2}+2k\pi$, for all $k\in\N$ and $i:=2(n_0+1)$, we get for all $k\in\N$ that 
	\[
	(1+x_k^2)^{-n_0}|(f\circ \varphi)^{2(n_0+1)}(x)|=(1+x_k^2)^{-n_0}|P(x_k)|\geq c x_k^2
	\]
	for some suitable positive constant $c$. Therefore, $|(f\circ\varphi)|_{2(n_0+1),n_0}=+\infty$. A contradiction.
\end{rem}

\section{Dynamics of composition operators acting on lcHs' continuously included in   $C^m(\R)$ for some $m\in\N_0\cup\{\infty\}$}

Throught this section $X$ always denotes a separable lcHs. Let $T\in \cL(X)$. The operator $T$ is said to be  \textit{supercyclic} if  there exists $x\in X$ whose projective orbit under $T$ is dense in $X$, that is the set $ \{\lambda T^nx \colon  n\in\N_0, \lambda\in\C\}$ is dense in $X$. Such an $x$ is said to be a \textit{supercyclic vector} of $T$. 
The operator $T$ is said to be \textit{hypercyclic} if there exists $x\in X$ whose orbit under $T$ is dense in $X$, that is the set ${\rm orb}(x, T) := \{T^nx \colon  n\in\N_0\}$ is dense in $X$. Such an $x$ is said to be a \textit{hypercyclic vector} of $T$. No power bounded operator is hypercyclic, but can be supercyclic. Every hypercyclic operator is always supercyclic. The converse is not true in general.
The operator $T$ is said to be \textit{cyclic} if there exists $x\in X$ whose span of the orbit under $T$ is dense in $X$, that is the space ${\rm span}({\rm orb}(x, T))$ is dense in $X$. Such an $x$ is said to be a \textit{cyclic vector} of $T$. Every supercyclic operator is always cyclic. The converse is not true in general.
The operator $T$ is said to be \textit{topological transitive} if for every pair of non-empty open subsets $U$ and $V$ of $X$, there is $n\in\N$ such that $T^n(U)\cap V\neq \emptyset$. An hypercyclic operator is always topological transitive. Thanks to the Birkhoff-transitivity theorem, the converse holds for Fr\'echet spaces (see \cite[Theorem 1.2]{BM}).

We recall an important result that we use in the following. For the proof see \cite[Proposition 1.26]{BM}.

\begin{prop}\label{critk}
	Let $X$ be a separable lcHs and let $T\in\cL(X)$ be supercyclic. Then either $\sigma_p(T')=\emptyset$ or $\sigma_p(T')=\{\lambda\}$, for some $\lambda\neq0$. In the latter case, ${\rm Ker}(T'-\lambda)$ has dimension $1$ and ${\rm Ker}(T'-\lambda)^n={\rm Ker}(T'-\lambda)$ for all $n\in\N_0$. Moreover, there exists a (closed) $T$-invariant hyperplane $X_0\subset X$ such that $T_0:=\lambda^{-1}T_{|{X_0}}$ is hypercyclic on $X_0$.
\end{prop}
An operator $T\in \cL(X)$ is said to be \textit{mixing} if for every pair of non-empty open subsets $U$ and $V$ of $X$, there is $N\in\N$ such that $T^n(U)\cap V\neq \emptyset$ for all $n\geq N$. Every mixing operator on a separable Fr\'echet space $X$ is hypercyclic. 
In particular, in such a case it is easy to see that  if the operator $T$ is mixing, then it is also \textit{hereditarily hypercyclic}, i.e.,  for any infinite set $I\su\N$ the family $\{T^n\,:\, n \in I\}$  is \textit{universal} which means that  for some $x\in X$ the set $\{T^n x\,:\, n\in I\}$ is dense in $X$.

An operator $T$ is said to be \textit{weakly hyperercyclic} if there exists $x\in X$ whose orbit under $T$ is weakly dense in $X$ (for the weak topology $\sigma(X,X'))$. Such a vector $x$ is called a \textit {weakly hypercyclic vector} for $T$. One defines in the same way \textit{weakly supercyclic operators} and \textit{weakly supercyclic vectors}. We refer the reader to \cite{BM,GP} for more details.

In the following, if $X$ is a lcHs of functions defined on $\R$, we denote by $\delta_a$ the linear functional on $X$ which maps every function of $X$ to its value at $a\in \R$. We observe that if $X$ is continuously included in $C^m(\R)$, for some $m\in\N_0\cup\{\infty\}$, then $\delta_a^{(j)}\in X'$ for all $j\leq m$ and $a\in \R$, where $\delta_a^{(j)}$ is defined by $\delta_a^{(j)}(f):=f^{(j)}(a)$. A function $\varphi:\R\to \R$ is said to be {\em strongly runaway} when for any compact subset $K$ of $\R$ there exists $n_0\in\N$ such that $\varphi_n(K)\cap K=\emptyset$ for all $n\geq n_0$.  It is standard to show that for an increasing function  $\varphi:\R\to\R$  the property to be strongly runaway  is equivalent to the fact that  $\varphi$ has no fixed points (cf. \cite[Lemma 4.1]{Pr2} for a proof in the smooth case).

We now can prove the following result. 

\begin{prop}\label{nwsup}
	Let $X$ be a separable lcHs such that $X\hookrightarrow C^1(\R)$ and $\delta_a, \delta_a^{(1)}$ are linearly independent  in $X'$ for every  $a\in\R$. Let $\varphi\in C^1(\R)$ be a symbol for $X$. If $\varphi(a)=a$ for some $a\in\R$ or $\varphi'(a)=0$ for some $a\in\R$, then $C_\varphi: X\to X$ is not weakly supercyclic.
\end{prop}
\begin{proof} Since $\varphi$ is a symbol for $X$, $C_\varphi\in \cL(X)$ and hence $C_\varphi\in \cL((X,\sigma(X,X'))$. 

	Suppose firstly that $\varphi'(a)=0$ for some $a\in\R$. This assumption implies that $C_{\varphi}(X)\subseteq {\rm Ker}\, \delta^{(1)}_a$. Then $C_\varphi$ cannot be even (weakly) cyclic.
	
	Now, suppose that $\varphi(a)=a$ for some $a\in\R$. This assumption implies that $C'_\varphi(\delta_a)=\delta_a$ and $C'_\varphi(\delta^{(1)}_a)=\varphi'(a)\delta^{(1)}_{\varphi(a)}=\varphi'(a)\delta^{(1)}_a$. Now we have to distinguish two cases. If $\varphi'(a)\neq 1$, then $C'_\varphi$ has two different eigenvalues. Applying Proposition \ref{critk} to $(X,\sigma(X,X'))$, we get that $C_\varphi$ is not weakly supercyclic. If $\varphi'(a)= 1$, then $\delta_a$ and $\delta^{(1)}_a$ are two different eigenvectors for $C_\varphi'$, thereby implying that ${\rm Ker}(C'_\varphi-I)$ has dimension greater or equal than two. Again, applying Proposition \ref{critk} to $(X,\sigma(X,X'))$, we get that $C_\varphi$ is not weakly supercyclic.
\end{proof}

\begin{cor}\label{wsimpmix}
	Let $X$ be a separable lcHs such that $X\hookrightarrow C^1(\R)$ and $\delta_a, \delta_a^{(1)}$ are linearly independent  in $X'$ for every  $a\in\R$. Let $\varphi\in C^1(\R)$ be a symbol for $X$.
	If $C_\varphi$ is weakly supercyclic on $X$, then $\varphi$ is strongly runaway and $\varphi'(x)>0$ for all $x\in\R$. 
\end{cor}
\begin{proof}
	By Proposition \ref{nwsup} we can assume that $\varphi'(x)\neq 0$ for every $x\in\R$. Since   Proposition \ref{nwsup} also implies that the function  $\varphi$ cannot have fixed points, it follows that $\varphi$ is  increasing  ($\varphi$ decreasing would have fixed points). Accordingly, the function $\varphi(x)-x$ has constant sign. From this it follows that all the orbits of $\varphi$ are divergent. Hence $\varphi$  is strongly runaway.
\end{proof}

Observe that in  Proposition \ref{nwsup}, we ask that $X\hookrightarrow C^1(\R)$. We now prove the following result which concerns the weaker assumption $X\hookrightarrow C(\R)$.

\begin{prop}\label{Psuper} Let X be a separable lcHs  such that $X\hookrightarrow C(\R)$ with dense inclusion and $\{\delta_a:\ \in \R\}$ is a linearly independent subset in $X'$.
	 Let $\varphi\in C(\R)$ be a symbol for $X$.
	If $C_\varphi$ is supercyclic on $X$, then $\varphi$ is  increasing and without fixed points. 
\end{prop}

\begin{proof}
		The density of the inclusion $X\hookrightarrow C(\R)$ permits us to consider only  the case $X=C(\R)$ in order to  prove that  $C_\varphi$ supercyclic implies $\varphi$   increasing and without fixed points.  
	
	If $\varphi$ is not injective, then there exist $a,b\in\R$ with  $a\neq b$ such that $\varphi(a)=\varphi(b)$. Then $C_\varphi(C(\R))\subseteq {\rm Ker}(\delta_a-\delta_b)$, thereby implying that $C_\varphi$ cannot be (even weakly) supercyclic.

	We  now suppose  that $\varphi$ has fixed points. If there are more than one, for instance $a,b\in\R$ with $a\not= b$, we have that $C'_\varphi(\delta_a)=\delta_a$ and $C'_\varphi(\delta_b)=\delta_b$. By Proposition \ref{critk} it follows that $C_\varphi$ cannot be (even weakly) supercyclic.
	
	Finally, we assume  that $\varphi$ has a unique fixed point, that without loss of generality we can suppose to be $0$. Since $\varphi$ is injective, the function $\varphi(x)-x$ has constant sign in $(-\infty,0)$ and in $(0,\infty)$. Suppose $\varphi$ to be increasing. If $\varphi(x)<x$ for all $x>0$, then $\varphi_n(x)\to 0$ as $n\to\infty$ for all $x>0$ and so $C_\varphi$ is not (even weakly) supercyclic by \cite[Theorem 8(ii)]{BJM}. Hence, $\varphi(x)>x$ for every $x>0$. In a similar way, we can show that $\varphi(x)<x$ for every $x<0$. Therefore, the function $\varphi$ is also surjective and $\varphi^{-1}(x)<x$ for every $x>0$ and $\varphi^{-1}(x)>x$ for every $x<0$.  Accordingly, $\varphi^{-1}\in C(\R)$. Since $C_\varphi: C(\R)\to C(\R)$ is supercyclic,  also $C_{\varphi^{-1}}: C(\R)\to C(\R)$ is supercyclic. Again we get a contradiction by \cite[Theorem 8(ii)]{BJM}. If $\varphi$ is not increasing, it must be decreasing since it is injective. The contradiction now follows taking into account of the fact that $C_{\varphi}: C(\R)\to C(\R)$ is supercyclic if, and only if, $C_{\varphi_2}: C(\R)\to C(\R)$ is supercyclic.
		\end{proof}

Combining our results with previous works of Bonet and Doma\'nski and Kalmes, we get the following result for dynamics of some composition operators. The proof of (i) below uses an unpublished argument due to Bonet and Doma\'nski which ensures that $C_\varphi$ is sequentially supercyclic in $\mathcal{A}(\R)$ for every analytic diffeomorphism  $\varphi$ in $\R$  without fixed points.

\begin{thm}\label{BDK} 
\begin{itemize}
	\item[(i)] If a composition operator $C_\varphi$ defined on $\mathcal{A}(\R)$ is weakly supercyclic, then $C_\varphi$ is mixing. If we assume in addition $\varphi$ to be surjective, then $C_\varphi$ is weakly supercyclic if, and only if, $C_\varphi$ is sequentially hypercyclic.
	\item[(ii)] A composition operator $C_\varphi$ defined on $C^{m}(\R)$ with $m\in \N\cup \{\infty\}$ is weakly supercyclic if, and only if, $C_\varphi$ is mixing.	
	\item[(iii)] A composition operator $C_\varphi$ defined on $C(\R)$ is supercyclic if, and only if, $C_\varphi$ is mixing.
\end{itemize}

	\end{thm}
	
\begin{proof} (i). By Corollary \ref{wsimpmix}, if $C_\varphi$ is weakly supercyclic then $\varphi$ must be runaway, $\varphi'(x)\neq 0$ for all $x\in\R$, and $\varphi$ does not have any fixed point.
 By \cite[Theorem 2.3]{BD2} $C_\varphi$ is then mixing. If $\varphi$ is also surjective, then $\varphi$ is an analytic diffeomorphism, and therefore $\varphi$ is analytically conjugate to $\psi(x)=x+1$ by \cite[Theorem 3.1]{BL}. We conclude since $C_\psi$ is sequentially hypercyclic by \cite[Theorem 3.6]{BD2}.
 
   (ii) It follows by combining Corollary \ref{wsimpmix}  with \cite[Corollary 4.2]{Kalmes2}.
   
    Finally, (iii) is a consequence of  Proposition \ref{Psuper} and \cite[Corollary 4.1]{Kalmes2}.
\end{proof}



\begin{prop}
	Let $\varphi(x)=x+d$, $d\neq 0$. The operator $C_\varphi$ is mixing on $\cO^m(\R)$ for any $m\in\N$, and  also on $\cO_M(\R)$.
\end{prop}

\begin{proof} Using Remark \ref{convom}
	it is easy to show that, for every $f\in \cD(\R)$, both  sequences $\{f\circ \varphi_n\}_{n\in\N}$ and $\{f\circ \varphi_{-n}\}_{n\in\N}$ are convergent  to 0 in $\cO^m(\R)$ for any $m\in\N$ and hence, in $\cO_M(\R)$. Now, by applying \cite[Theorem 1.6]{BM} with $n_k=k$, $\cD_1=\cD_2=\cD(\R) $ and $S=C_\varphi$, we get that $C_\varphi$ is  mixing in $\cO^m(\R)$ for any $m\in\N$ and also in $\cO_M(\R)$.	
\end{proof}

The basic example of the translations gives us mixing operators on both spaces $\cO^m(\R)$, $m\in\N$, and $\cO_M(\R)$. This should be compared with \cite[Theorem 2.1]{FeGaJo}  and \cite[Corollary 2.2(1)]{FeGaJo}, which shows how the situation differs when  $\{f\circ \varphi_n\}_{n\in\N}$ is not convergent to $0$ when $f\in \cD(\R)$ and $\varphi(x)=x+d$. We do not know if $C_\varphi:\cO_M(\R)\to \cO_M(\R)$ or $C_\varphi:\cO^m(\R)\to \cO^m(\R)$ is mixing when $\varphi$ is strongly runaway and $\varphi'(x)>0$ for all $x\in\R$.

\section{Ergodic properties
	of composition operators acting on lcHs' continuously included
	in  $C^m(\R)$ for some $m\in\N_0\cup\{\infty\}$}

\label{sot}	
\subsection{Strong operator topology on operators on spaces of
	differentiable functions}
We start by proving that the superposition operator is continuous.
\begin{prop}
	\label{superposition}
	Let $m\in \N_0\cup\{\infty\}$ and let $f\in C^{m}(\R)$. Then the superposition operator $S_{f}:\ C^{m}(\R)\to C^{m}(\R)$, $\varphi\mapsto f\circ\varphi$, is continuous.
\end{prop}
\begin{proof}
	We prove the statement by induction on $m$. Let $f\in C(\R)$ and let $\{\varphi_n\}_{n\in\N}\subset C(\R)$ be a convergent sequence in $C(\R)$ to some $\varphi\in C(\R)$.  We will show that the sequence $\{f\circ \varphi_n\}_{n\in\N}$ is convergent to 
	$f\circ \varphi$ in the topology $\tau_c$ of $C(\R)$. To do this, we first observe that for all $k\in \N$ there exists $M\in\N$ such that $\varphi_n([-k,k])\subseteq [-M,M]$ for all $n\in\N$. Indeed, there exists $n_0\in\N$ such that $n\geq n_0$ implies $|\varphi_n(x)-\varphi(x)|\leq 1$ for each $x\in [-k,k]$. On the other hand, $\varphi([-k,k])$ and $\varphi_i([-k,k])$, $i=1,\ldots,n_0-1$, are compact subsets of $\R$ as $\varphi,\varphi_i\in C(\R)$, $i=1,\ldots,n_0-1$, and hence, $\varphi([-k,k])\subset [-M',M']$ and  $\varphi_i([-k,k])\subset [-M',M']$,  $i=1,\ldots,n_0-1$, for some $M'>0$. Therefore, $\varphi_n([-k,k])\subset [-M'-1,M'+1]$ for all $n\in\N$. 
	
	Since $f\in C(\R)$, $f$ is uniformly continuous in $[-M,M]$ and hence, for a fixed $\varepsilon>0$, there exists $\delta>0$ such that $|f(x)-f(y)|<\varepsilon$ for each $x,y\in[-M,M]$ with $|x-y|<\delta$. On the other hand, $\sup_{x\in [-k,k]}|\varphi_n(x)-\varphi(x)|\to 0$ for $n\to\infty$ and hence, there exists  $n_0\in\N$ such that $\sup_{x\in [-k,k]}|\varphi_n(x)-\varphi(x)|<\delta$ for all $n\geq n_0$. Since $\varphi_n([-k,k])\subset [-M,M]$ for all $n\in\N$,  we get that $|(f\circ \varphi_n)(x)-(f\circ\varphi)(x)|<\varepsilon$ for all $n\geq n_0$ and $x\in [-k,k]$.

	Assume now the hypothesis is  true for $m-1$. Let $f\in C^{m}(\R)$ and let $\{\varphi_n\}_{n\in\N}\subset C^{m}(\R)$ be a convergent sequence in $C^m(\R)$ to some $\varphi\in C^{m}(\R)$. 
	For all $n\in\N$, 
	for every $x\in\R$ and $i=0,\dots,m$, Fa\`{a} di Bruno formula gives:
	\begin{align}\label{eq.derivate2}
	&(S_f(\varphi_n) )^{(i)}(x)=\nonumber\\
	&=\sum\frac{i!}{k_1!k_2!\ldots k_n!}f^{(k)}(\varphi_n(x))\left(\frac{\varphi_n'(x)}{1!}\right)^{k_1}\left(\frac{\varphi_n''(x)}{2!}\right)^{k_2}\ldots \left(\frac{\varphi_n^{(i)}(x)}{i!}\right)^{k_i},
	\end{align}
	where the sum is extended over all $(k_1,k_2,\ldots,k_i)\in \N_0^i$ such that $k_1+2k_2+\ldots+ ik_i=i$ and $k_1+k_2+\ldots+k_i=k$. We get from \eqref{eq.derivate2}, the inductive hypothesis  and the continuity of the product in $C(\R)$ that $\{(S_f(\varphi_n) )^{(i)}\}_{n\in\N}$ is convergent to $(S_f(\varphi) )^{(i)}$ in $C(\R)$ for any $0\leq i\leq m$, thereby completing  the proof. 
	\end{proof}

\begin{cor}
	\label{corsup}
	Let $m\in \N_0\cup \{\infty\}$ and $X$ be a barrelled lcHs continuously included in $C^m(\R)$ and satisfying the following properties:
	\begin{itemize}
		\item[(i)] The identity $p_1(x)=x\in X$;
		\item[(ii)]  $X$ is closed under composition;
		\item[(iii)] A sequence $\{\varphi_n\}_{n\in\N}\subset X$ is convergent in $X$ to some  $\varphi\in X$ if, and only if, $\{\varphi_n\}_{n\in\N}$ is bounded in $X$ and $\lim_{n\to\infty} \varphi_n=\varphi$ in $C^m(\R)$. 
	\end{itemize}
	Let $\{\varphi_n\}_{n\in\N}\subset X$ be a sequence of symbols for $X$. Then the sequence $\{C_{\varphi_n}\}_{n\in\N}$ of composition operators  is convergent in $\cL_s(X)$ if, and only if, $\{\varphi_n\}_{n\in\N}$ is convergent in $X$ and $\{C_{\varphi_n}\}_{n\in\N}$ is an equicontinuous sequence in $\cL(X)$. 
\end{cor}

\begin{proof}
	If there exists $C\in \cL(X)$ such that $\{C_{\varphi_n}(f)\}_{n\in\N}$ is convergent to $C(f)$ for any $f\in X$, then  for $f=p_1$ we get that $\{\varphi_n\}_{n\in\N}$ converges in $X$ to $C(p_1)$. Since $X$ is barrelled,  we also obtain that  $\{C_{\varphi_n}\}_{n\in\N}$ is an equicontinuous sequence in $\cL(X)$. 

	The converse follows from Proposition \ref{superposition}. Indeed, if the sequence $\{\varphi_n\}_{n\in\N}$ is convergent in $X$ to some $\varphi\in X$, then by (iii) the sequence $\{\varphi_n\}_{n\in\N}$ is convergent in $C^m(\R)$ to  $\varphi$ and so, by Proposition \ref{superposition} we have that  $C_{\varphi_n}(f)=f\circ \varphi_n\to f\circ\varphi$ in $C^m(\R)$ as $n\to\infty$ for any fixed $f\in X$.  On the other hand,  the sequence $\{C_ {\varphi_n}(f)\}_{n\in\N}$ is bounded in $X$ for any fixed $f\in X$, being $\{C_{\varphi_n}\}_{n\in\N}$  an equicontinuous sequence in $\cL(X)$. Therefore, by (iii) we conclude that $C_{\varphi_n}(f)=f\circ \varphi_n\to f\circ\varphi$ in $X$ as $n\to\infty$ for any fixed $f\in X$. The equicontinuity of $\{C_{\varphi_n}\}_{n\in\N}$ implies that $C_\varphi\in\cL(X)$. So, $C_{\varphi_n}\to C_\varphi$ in $\cL_s(X)$ as $n\to\infty$.
\end{proof}

\begin{rem}
	\label{montel} If $X$ is a Montel lcHs space which is continuously included in  $C^m(\R)$, with $m\in\N_0\cup\{\infty\}$, then $X$ clearly satisfies condition (iii) of Corollary \ref{corsup}.
\end{rem}
\begin{rem}
	\label{stable}
	If the sequence $\{\varphi_n\}_{n\in\N}$ is bounded in $C^m(\R)$ for some $m\in \N_0\cup \{\infty\}$, then  for all $k\in\N$ there exists $M>0$ such that  $\sup_{|x|\leq k}|\varphi_n^{(j)}(x)|< M$  for all $0\leq j\leq m$ and  $n\in\N$, i.e., $\varphi_n^{(j)}([-k,k])\subseteq [-M,M]$ for all $0\leq j\leq m$ and  $n\in\N$.  Hence,  \cite[Theorem 3.13]{Kalmes} for the particular case of composition operators reads as follows: $C_\varphi$ is power bounded in $C^m(\R)$ if, and only if, $\{\varphi_n\}_{n\in\N}$ is bounded in $C^m(\R)$.
\end{rem}
\subsection{Polynomial symbols}

We collect some results on  mean ergodicity of $C_\varphi$. 

\begin{prop}\label{P.ME0} Let $X$ be  a lcHs  which is closed under composition, continuously inluded in $C^m(\R)$ for some $m\in\N_0\cup\{\infty\}$ and contains  the function $p_1(x):=x$, for  $x\in\R$. If $\varphi\in X$ is a symbol for $X$ and
	$C_\varphi$ is  mean ergodic on $X$, then the following properties are satisfied:
	\begin{enumerate}
		\item The  sequence $\{\varphi_{[n]}\}_{n\in\N}$ converges in $C^m(\R)$;
		\item The sequence $\{\frac{\varphi_{n}}{n}\}_{n\in\N}$ converges to $0$ in $C^m(\R)$.
	\end{enumerate}
	If $C_\varphi$ is power bounded on $X$, condition $(2)$ above is satisfied.
\end{prop}

\begin{proof} Assume that $C_\varphi$ is mean ergodic on $X$. 
	
	(1) By assumption there exists $P\in X$ such that $(C_\varphi)_{[n]}\to P$ in $\cL_s(X)$. Since the function $p_1$ belongs to $X$, it follows that $(C_\varphi)_{[n]}(p_1)=\varphi_{[n]}\to Pp_1:=\psi$ in $X$. Hence,  $\{\varphi_{[n]}\}_{n\in\N}$ is also convergent in $C^m(\R)$ to $\psi$. Therefore,   condition (1) follows.

	(2) Since
	\[
	\frac{\varphi_{n}}{n}=\varphi_{[n]}-\frac{n-1}{n}\varphi_{[n-1]},\quad n\geq 2,
	\]
	the result immediately follows from part (1).
	
	Finally, if $C_\varphi$ is power bounded, then the sequence $\{C_{\varphi_n}(p_1)\}_{n\in\N}=\{\varphi_n\}_{n\in\N}$  is bounded in $X$ and hence, bounded in $C^m(\R)$, thereby implying that  (2) is satisfied. 
\end{proof}

\begin{thm}
	\label{E.PB}
		Let $\varphi$ be a polynomial and let $X\in\{\cO_M(\R), C^\infty(\R), \mathcal{A}(\R)\}\cup\{\cO^m(\R):\ m\in\N_0\}\cup\{C^m(\R):\ m\in\N_0\}$. Then the composition operator $C_\varphi$ is mean ergodic on $X$ if, and only if,  $C_\varphi$ is power bounded if, and only if, $\varphi(x)=ax+b$, with either $|a|<1$, $a=-1$ or $a=1$ and $b=0$. In case $|a|<1$, we even have that the sequence of iterates $\{C_{\varphi_n}\}_{n\in\N}$ converges in $\cL_s(X)$. 
\end{thm}	

\begin{proof} Assume that either the degree of $\varphi$ is greater or equal than two or $\varphi(x)=ax+b$ with $|a|>1$. Then there exists   $n_0\in \N$ such that $|x|\geq n_0$ implies 
	$|\varphi(x)|>|x| +1$. So, it follows that  $|\varphi_n(n_0)|\geq n_0+n$ for all $n\in\N$. Therefore, $C_\varphi$ is neither power bounded nor mean ergodic on $X$ by Proposition \ref{P.ME0}. If $\varphi(x)=x+b$, with $b\neq 0$, then $\varphi_n(0)=nb$, and again Proposition \ref{P.ME0} yields that $C_\varphi$ cannot be power bounded or mean ergodic on $X$.

	Let $\varphi(x):=ax+b$ for $x\in\R$ and some $a,b\in\R$. The case $a=1$, $b=0$ gives the identity. The case $a=-1$ gives $\varphi_2=\varphi$, and  hence, $C_{\varphi}$ is idempotent. In both cases the statement is trivial. Also it is trivial when $a=0$. Indeed, in this case $C_\varphi$ is the evaluation at $b$ and satisfies $C_{\varphi_n}=C_\varphi$ for all $n\in\N$.

	Let  $|a|<1$, with $a\neq 0$. In this case, for $\phi(x):=x+b/(a-1)$, and $\psi(x):=ax$, we have $C_\varphi=C_{\phi^{-1}\circ\psi\circ\phi}=C_{\phi}C_\psi C_{\phi^{-1}}$. Hence, power boundedness or mean ergodicity of  $C_\varphi$ is equivalent to that of $C_\psi$. Thus, we can assume $b=0$. In this case it is easy to show that the sequence $\{\varphi_n\}_{n\in\N}$ is convergent to the constant function $0$ in all the spaces we are considering and hence, bounded there. Accordingly, the operator $C_\varphi$ is power bounded in every space $X$ (see Theorem \ref{T.PowerB} for $X$ equals to $\cO_{M}(\R)$ or equals to $\cO^m(\R)$ with $m\in\N$; see Remark \ref{stable} for $X$ equals to $C^m(\R)$ with $m\in\N_0\cup\{\infty\}$; for $X=\mathcal{A}(\R)$ see \cite{BD}).  Since every space $X$ satisfies the assumptions (i)$\div$(iii) of Corollary \ref{corsup}, the conclusion  follows from Corollary \ref{corsup}.
\end{proof}

We point out that if $\varphi$ is a polynomial of even degree and with no fixed points, like $\varphi(x)=x^2+1$ for $x\in\R$, then the composition operator $C_\varphi$ is power bounded on $\cS(\R)$ (see \cite[Theorem 3.7]{FeGaJo}). {The key point to make the difference is that in $\cS(\R)$ the polynomials are symbols, but they do not belong to $\cS(\R)$.} We see below that in case of analytic functions we can go further than polynomial symbols.

\begin{rem}
Bonet and Domanski showed in \cite{BD} that a nontrivial composition operator $C_\varphi$ (i.e., $\varphi_2\neq p_1$)  defined on $\mathcal{A}(\R)$ is power bounded if, and only if, $\{\varphi_n\}_{n\in\N}$ is convergent in $\mathcal{A}(\R)$ to a (unique) fixed point $a$ of $\varphi$. Hence, from Corollary \ref{corsup} we get that such an operator satisfies that $\{C_{\varphi_n}\}_{n\in\N}$ is convergent in $\cL_b(\mathcal{A}(\R))$.
\end{rem}

\subsection{Monotonic symbols}

\begin{prop}\label{P.ME2} Let $X$ be a lcHs which is   closed under composition, continuously included in $ C(\R)$ and contains the function $p_1$. Let $\varphi\in X$ be a symbol for $X$ such that there exists $\lim_{n\to\infty}\varphi_n(x)$ for all $x\in\R$.  If  $C_\varphi$ is mean ergodic on $X$, then the set ${\rm Fix}(\varphi)$ of fixed points of $\varphi$ is a non empty closed interval of $\R$.
\end{prop}

\begin{proof} 
	Since $C_\varphi$ is mean ergodic on $X$, we can apply 
	Proposition \ref{P.ME0}(1) to get that the sequence $\{\varphi_{[n]}\}_{n\in\N}$ converges in $C(\R)$ to some function $\psi\in C(\R)$. Since there exists $\lim_{n\to\infty}\varphi_n(x)$ for every $x\in\R$,  by Ces\`aro's theorem we necessary have
	\[
	\psi(x)=\lim_{n\to\infty}\varphi_n(x)
	\]
	for every $x\in\R$. So, by the fact that $\varphi$ is a continuous function,  we also have for every $x\in\R$ that
	\[
	\psi(x)=\lim_{n\to\infty}\varphi_n(x)=\lim_{n\to\infty}\varphi(\varphi_{n-1}(x))=\varphi(\psi(x)).
	\]
	This means that  $\psi(x)\in {\rm Fix}(\varphi)$ for all $x\in\R$ and hence,   ${\rm Fix}(\varphi)={\rm Im}\psi$. Since ${\rm Fix}(\varphi)$ is always a closed set and    $\psi\in C(\R)$, it follows that ${\rm Fix}(\varphi)$  is a non empty closed interval of $\R$. 
\end{proof}

\begin{prop}\label{P.ME1} Let $X$ be a lcHs which is closed  under composition, continuously included in  $C(\R)$ and contains the function $p_1$. Let $\varphi\in X$ be a symbol for $X$.
	If there exist $\beta>0$  such that $\varphi$ is a non decreasing function on $[\beta,+\infty)$ (on $(-\infty,\beta]$, respectively),  and $\varphi(x)>x$  ($\varphi(x)<x$) for every $x\geq\beta$ ($x\leq \beta$, respectively), then $C_\varphi$ is not mean ergodic on $X$.
\end{prop} 

\begin{proof} We consider only the case in which $\varphi$ is a non decreasing function on $[\beta,+\infty)$. In the other case the proof follows by arguing in a similar way.

	We first observe that  $\{\varphi_n(\beta)\}_{n\in\N}$  is a non decreasing sequence, but not a bounded sequence. Indeed, from $\varphi(\beta)>\beta$ follows that there exists $\delta>0$ such that $\varphi(\beta)>\beta +\delta$ and hence, $\varphi_n(\beta)>\beta +n\delta$ for all $n\in\N$. Accordingly, $\frac{\varphi_n(\beta)}{n}>\frac{\beta}{n}+\delta>\delta$ for all $n\in\N$. Hence, $\{\frac{\varphi_n(\beta)}{n}\}_{n\in\N}$ does not converge to $0$. So, by Proposition \ref{P.ME0} we can conclude that $C_\varphi$ is not mean ergodic on $X$.
\end{proof}

\begin{cor}\label{C.ME} Let $X$ be a  lcHs which is closed  under composition, continuously included in  $C^1(\R)$ and contains the function $p_1$.  Let $\varphi\in X$ be a symbol for $X$, $\varphi\neq p_1$.
	If $\varphi$ is a non decreasing function  and $C_\varphi$ is mean ergodic on $X$, then there is $a\in\R$ such that ${\rm Fix}(\varphi)=\{a\}$ and $\{\varphi_n\}_{n\in\N}$ converges  in $C(\R)$ to the constant function $\psi(x):=a$ for $x\in\R$.
\end{cor}

\begin{proof} 
	 Since $\varphi$ is a non decreasing function, there exists $\lim_{n\to\infty}\varphi_n(x)$ for all $x\in \R$. Indeed, for a fixed $x\in \R$, we have either $\varphi(x)\leq x$ or $\varphi(x)\geq x$ from which it follows that $\{\varphi_n(x)\}_{n\in\N}$ is either a non increasing  sequence or a non decreasing sequence, being $\varphi$  a non decreasing function.  Since $C_\varphi$ is mean ergodic, we can then apply Proposition \ref{P.ME2} to conclude that the sequence $\{\varphi_{[n]}\}_{n\in\N}$ converges in $X$ (and hence, in $C^1(\R)$) to some $\psi\in X$  and  the set ${\rm Fix}(\varphi)=\psi(\R)$ is a closed interval of $\R$. Since $\varphi$ is not the identity function on $\R$, necessarily $ {\rm Fix}(\varphi)$ is a proper subset of $\R$. Moreover,   from the proof of Proposition \ref{P.ME2} it also follows  that $\psi(x)=\lim_{n\to\infty}\varphi_n(x)$ for every $x\in\R$ and hence, the sequence $\{\varphi_n(x)\}_{n\in\N}$ is a convergent sequence for every $x\in\R$.

	Let assume ${\rm Fix}(\varphi)$ to be bounded below and let $a$ be the infimum of ${\rm Fix}(\varphi)$. Clearly, $a\in{\rm Fix}(\varphi)$. On the other hand, since $C_\varphi$ is mean ergodic,  Proposition \ref{P.ME1}  implies that $\varphi(x)>x$ for all $x<a$. Accordingly, taking into account that $\varphi$ is a non decreasing function, it follows for every $x<a$ that $\{\varphi_n(x)\}_{n\in\N}$ is a non decreasing sequence such that $\varphi_n(x)\leq \varphi_n(a)=a$ for all $n\in\N$. Hence,  $\psi(x)=\lim_n \varphi_n(x)\leq a$ for all $x< a$. But, we also have for every $x< a$ that $\psi(x)=\lim_{n\to\infty}\varphi(\varphi_{n-1}(x))=\varphi(\psi(x))$ and hence, $\psi(x)\in {\rm Fix}(\varphi)$, thereby implying that $\psi(x)\geq a$. So,  $\psi(x)=a$ for all $x<a$. If ${\rm Fix}(\varphi)$ is not a single point, then there is $\delta>0$ such that $[a,a+\delta)\subseteq {\rm Fix}(\varphi)$. For all $x\in [a,a+\delta)$ we have $\psi(x)=\lim_n \varphi_{[n]}(x)=x$. This contradicts the differentiability of $\psi$ at $a$.

	The proof of the case ${\rm Fix}(\varphi)$  bounded above is completely analogous.

	Now, we can assume that   $a=0$ without loss of generality. So,  Proposition \ref{P.ME1} yields $|\varphi(x)|\leq |x|$ for all $x\in\R$. From this we get immediately that $\lim_{n\to\infty}\varphi_n(x)=0$ uniformly on compact subsets of $\R$. 
\end{proof}


\begin{thm}
\label{onefixed}
Let $X$ be a Montel lcHs which is closed under composition, continuously included in $C^1(\R)$ and contains the function $p_1$. Assume that the following condition is also satisfied:
\begin{itemize}

\item[(*)]   The composition operator $C_\varphi: X\to X$ is power bounded if, and only if, $\{\varphi_n\}_{n\in\N}$ is bounded in $X$ for any $\varphi\in X$ symbol for $X$.
\end{itemize}
Let $\varphi\in X$ be monotonic, $\varphi_2\neq p_1$.  Then $C_\varphi$ is power bounded if, and only if, $\varphi$ has an attracting fixed point $a$ and $\{C_{\varphi_n}\}_{n\in\N}$ is convergent to $C_a$ in $\cL_s(X)$. 
\end{thm}

\begin{proof} We first observe that if $C_\varphi$ is power bounded, then the sequence $\{\varphi_n\}_{n\in\N}$ is  bounded in $X$, being $\varphi_n=C_{\varphi_n}p_1$ for all $n\in\N$. 
	
 Assume first that $\varphi$ is non decreasing.  Since $X$ is Montel and $C_\varphi $ is power bounded, $C_\varphi$ is also (uniformly) mean ergodic. Thus, by Corollary \ref{C.ME} the sequence $\{\varphi_n\}_{n\in\N}$ necessarily  converges in $C(\R)$ to a unique fixed point $a$. Now, the fact that $X$ is Montel together with (*) imply   that $\{\varphi_n\}_{n\in\N}$ is also convergent in $X$ to $a$. Therefore, we can apply   Corollary \ref{corsup} to conclude  the convergence of $\{C_{\varphi_n}\}_{n\in\N}$ to $C_a$ in $\cL_s(X)$, taking into account that $C_\varphi$ is power bounded. Vice versa, if $\{C_{\varphi_n}\}_{n\in\N}$ is convergent in $\cL_s(X)$ to $C_a$, then $C_\varphi$ is power bounded, and we conclude. 

Assume now that $\varphi$ is decreasing. Since $X$ is barrelled, by Banach-Steinhaus theorem we have  that any operator $T\in \cL(X)$ is power bounded if, and only if, $T^2$ is.  Accordingly, $C_\varphi$ is power bounded if, and only if, $C_{\varphi_2}$ is. Since $\varphi_2$ is increasing,  by the first part of the proof we get that $C_\varphi$ is power bounded if, and only if, $\varphi_2$ has only one fixed point $a$ and $\{C_{\varphi_{2n}}\}_{n\in\N}$ is convergent to $C_a$ in $\cL_s(X)$. It follows that also  $\varphi$ has only one fixed point $a$.  On the other hand,  the convergence of $\{C_{\varphi_{2n}}\}_{n\in\N}$ to $C_a$ in $\cL_s(X)$ implies the  convergence of $\{\varphi_{2n}\}_{n\in\N}$  to the constant function $a$ in in $C(\R)$. By applying this to any $\varphi([\alpha,\beta])$, with $\alpha,\beta\in\R$, we obtain the convergence  of $\{\varphi_n\}_{n\in\N}$ to the constant function $a$ in $C(\R)$. Taking account of the fact that $\{\varphi_{n}\}_{n\in\N}$ is a bounded sequence of $X$ as $C_\varphi$ is power bounded, we can apply  Corollary \ref{corsup} to conclude that $\{C_{\varphi_n}\}_{n\in\N}$ is convergent to $C_a$ in $\cL_s(X)$.  
\end{proof}

\section{Power boundedness of  composition operators on $\cO_M(\R)$ and on $\cO^m(\R)$}



A characterization of the power boundedness of $C_\varphi$ when it acts continuously on $\cS(\R)$ has been proved in \cite{FeGaJo}. Precisely, in  \cite[Proposition 4]{FeGaJo} it has been shown that if $\varphi\in C^\infty(\R)$ is  a symbol for $\cS(\R)$, then the composition operator $C_\varphi$ is power bounded if, and only if, the following conditions are satisfied:
	\begin{enumerate}
		\item For all $j\in\N_0$ there exist $C,p>0$ such that, for every $x\in\R$ and $n\in\N$
		\begin{equation*}\label{symbpb1}
		|\varphi_n^{(j)}(x)|\leq C(1+\varphi_n(x)^2)^p;
		\end{equation*}
		\item There exists $k>0$ such that 
		\begin{equation*}\label{symbpb2}
		|\varphi_n(x)|\geq |x|^{\frac{1}{k}}, \quad \forall \,|x|\geq k, \, \forall n\in \N.
		\end{equation*}
	\end{enumerate}

In order to treat the case of the space $\cO_M(\R)$ we first establish the following result.

\begin{thm}\label{T.PowerBm} Let $m\in\N$ and $\varphi\in \cO^m(\R)$. Then the following properties are equivalent:
	\begin{enumerate}
		\item $C_{\varphi}$ is power bounded on $\cO^m(\R)$;
		\item The sequence $\{\varphi_n\}_{n\in\N}$ is bounded in $\cO^{m}(\R)$;
		\item For all $0\leq i\leq m$ there exist $C,p>0$ such that $|\varphi_n^{(i)}(x)|\leq C(1+x^2)^p$ for all $x\in\R$ and $n\in\N$.
	\end{enumerate}
\end{thm}
\begin{proof} 
	(1)$\Rightarrow$(2). The operator $C_{\varphi}$ is power bounded on $\cO^m(\R)$. Accordingly, if we consider the function $p_1(x):=x$, for $x\in \R$, which belongs to $\cO^m(\R)$, then by assumption the sequence $\{C_{\varphi}^n p_1\}_{n\in\N}=\{\varphi_n\}_{n\in\N}$ is  bounded in $\cO^m(\R)$.
	
	(2)$\Rightarrow$(1).   We claim that for all $r\in\N$ there exists $r'\in\N$ such that $C_\varphi\colon \cO^m_r(\R)\to \cO_{r'}^m(\R)$ is power bounded, thereby obtaing that  the sequence $\{f\circ\varphi_n\}_{n\in\N}$ is bounded in $\cO_{r'}^m(\R)$ for all $f\in \cO_{r}^m(\R)$.
	Indeed, for fixed $r\in\N$ and $f\in \cO^m_r(\R)$, we observe for every $x\in\R$  and $n\in\N$   that 
	\[
	(C_\varphi^nf)^{(i)}(x)=\sum\frac{i!}{k_1!k_2!\ldots k_i!}f^{(k)}(\varphi_n(x))\left(\frac{\varphi_n'(x)}{1!}\right)^{k_1}\left(\frac{\varphi_n''(x)}{2!}\right)^{k_1}\ldots \left(\frac{\varphi_n^{(i)}(x)}{i!}\right)^{k_i}, \]
	where the sum is extended over all $(k_1,k_2,\ldots, k_i)\in \N_0^i$ such that $k_1+2k_2+\ldots +ik_i=i$, with $k=k_1+\ldots k_i$. Since by assumption on $\{\varphi_n\}_{n\in\N}$ (i.e., $\{\varphi_n\}_{n\in\N}$ is a bounded sequence of $\cO^m(\R)$), there exist $D>0$ and $p\in\N$ such that $|\varphi_n^{(j)}(x)|\leq D(1+x^2)^p$ for every $x\in\R$, $j=0,\ldots,m$ and $n\in\N$, and hence, $(1+\varphi_n(x)^2)\leq D' (1+x^2)^{2p}$ for every $x\in\R$ and $n\in\N$, and a suitable constant $D'>0$, it follows for every $x\in\R$, $i=0,\ldots,m$ and $n\in\N$ that 
	\begin{align*}
	|(C^n_\varphi f)^{(i)}(x)|&\leq \sum\frac{i!}{k_1!k_2!\ldots k_i!}|f^{(k)}(\varphi_n(x))|\left|\frac{\varphi'_n(x)}{1!}\right|^{k_1}\left|\frac{\varphi''_n(x)}{2!}\right|^{k_2}\ldots \left|\frac{\varphi_n^{(i)}(x)}{i!}\right|^{k_i}\\
	& \leq  \sum\frac{i!}{k_1!k_2!\ldots k_i!}|f|_{m,r}(1+\varphi_n(x)^2)^r\frac{D^k(1+x^2)^{pk}}{1!^{k_1}2!^{k_2}\ldots i!^{k_i}}\\&\leq  C_i |f|_{m,r}(1+x^2)^{p(2r+i)}\leq C |f|_{m,r}(1+x^2)^{p(2r+m)},
	\end{align*}
	where $C_i:=(D')^r\sum \frac{i!}{k_1!k_2!\ldots k_i!}\frac{D^k}{1!^{k_1}2!^{k_2}\ldots i!^{k_i}}$ and $C:=\max_{i=0,\ldots,m}C_i$. Accordingly,  if  $r':=p(2r+m)\in\N$, we get for all $n\in\N$ that 
	\begin{align*}
	|C^n_\varphi f|_{m,r'}\leq C |f|_{m,r}.
	\end{align*}
	Since $r\in\N$ and  $f\in\cO_r^m(\R)$ are arbitrary, the claim is proved.
	\\[0.1cm]
	We are now able to show that $C_\varphi$ is power bounded. Indeed, for a fixed $f\in \cO^m(\R)$,  there exists $r\in\N$ such that $f\in \cO_r^m(\R)$. By the proof above, it follows that  the sequence $\{C^n_\varphi f\}_{n\in\N}$ is bounded in $\cO^m_{r'}(\R)$ for some $r'\geq r$ and hence in  $\cO^m(\R)=\ind_{\stackrel{h}{\to}}\cO^m_{h}(\R)$, being it  an (LB)-space. So, we can conclude that $C_\varphi$ is power bounded on $\cO^m(\R)$ thanks to the Banach-Steinhaus theorem.

	$(2)\Leftrightarrow(3)$. Follows as $\cO^m(\R)$ is a regular (LB)-space.	
\end{proof}

Thanks to Theorem \ref{T.PowerBm}, we easily obtain the following characterization of the power boundedness of $C_\varphi$ when it acts in $\cO_M(\R)$. 

\begin{thm}\label{T.PowerB} Let $\varphi\in \cO_M(\R)$. Then the following properties are equivalent:
	\begin{enumerate}
		\item $C_{\varphi}$ is power bounded on $\cO_M(\R)$;
		\item $C_{\varphi}$ is power bounded on $\cO^m(\R)$ for all $m\in\N$;
		\item The sequence $\{\varphi_n\}_{n\in\N}$ is bounded in $\cO_{M}(\R)$;
		\item For all $i\in \N_0$ there exist $C,p>0$ such that $|\varphi_n^{(i)}(x)|\leq C(1+x^2)^p$ for all $x\in\R$ and $n\in\N$.
	\end{enumerate}
\end{thm}
\begin{proof} 
	Since $\cO_M(\R)$ is the projective limit of the sequence $(\cO^m(\R))_{m\in\N}$ of (LB)-spaces and it is barrelled, the Banach-Steinhaus theorem yields that $C_{\varphi}$ is power bounded if, and only if, $\{f\circ \varphi_n\}_{n\in \N}$ is bounded in $\cO^m(\R)$ for any $m\in\N$ and any $f\in \cO_M(\R)$. So, the conclusion follows now immediately from Theorem \ref{T.PowerBm}.
\end{proof}
From Remark \ref{stable}, Theorems \ref{T.PowerBm} and \ref{T.PowerB} we get immediately the following fact.

\begin{prop} 
	\label{pbs}
	Let $\varphi\in \cO_M(\R)$ ($\varphi\in \cO^m(\R)$, for $m\in\N$, resp.). Then $C_\varphi$ is power bounded  in $C^\infty(\R)$ (in $C^m(\R)$, resp.) whenever  $C_\varphi$ is power bounded in $\cO_M(\R)$ (in $\cO^m(\R)$, resp.).
\end{prop}

Combining Theorem \ref{onefixed} with Theorem \ref{T.PowerB}, we immediately obtain the following characterization, that should be compared with \cite[Theorem 3.13]{Kalmes}. 

\begin{cor} Let $\varphi\in \cO_M(\R)$ be an increasing function. Then $\{C_{\varphi_n}\}_{n\in\N}$ is convergent in $\cL_s(\cO_M(\R))$ if, and only if, $C_{\varphi}$ is power bounded on $\cO_M(\R)$ if, and only if,  $\{\varphi_n\}_{n\in\N}$ is bounded in $\cO_M(\R)$.
\end{cor}

{\bf Acknowledgements.} This paper was completed during the visit of the third author to the Univ\`{e}rsitat Politecnica de Val\`{e}ncia. He expresses his gratitude to Jos\'e Bonet and his collaborators for their kind hospitality and for their valuable suggestions and comments. 

The authors would like to thank the referee for his/her careful reading of the work and his/her valuable suggestions, which certainly improved the work. 

\end{document}